\documentclass[11pt,a4paper]{article}
\usepackage[english]{babel}
\usepackage{amsfonts}
\usepackage{amsmath}
\usepackage{amssymb}
\usepackage{bbm}
\usepackage{epsfig}
\usepackage{color}
\usepackage{enumerate}
\usepackage{amsthm}
\newtheorem{theorem}{Theorem}
\newtheorem{lemma}[theorem]{Lemma}
\newtheorem{corollary}[theorem]{Corollary}

\newtheorem{remark}[theorem]{Remark}

\def \dt{h}
\def \vrho{\varrho}
\def \veps{\varepsilon}

\def \DL{\Delta L}
\def \w{\boldsymbol{\omega}}
\def \tauv{\boldsymbol{\tau}}

\newcommand{\bR}{\mathbb{R}}

\renewcommand{\Re}{\operatorname*{Re}}
\renewcommand{\Im}{\operatorname*{Im}}

\newcommand{\tol}{\mathrm{tol}}
\newcommand{\bone}{\mathbbm{1}}
\newcommand{\mi}{\mathrm{i}}
\newcommand{\opt}{\text{opt}}

\def \bv{{\mathbf{b}}}
\def \cv{{\mathbf{c}}}
\def \Wv{{\mathbf{W}}}
\def \A{{\cal O}\!\!\iota}
\newcommand{\omegav}{\boldsymbol{\omega}}

\newcommand{\Deltav}{\boldsymbol{\Delta}}

\newcommand{\Iv}{{\mathbf{I}}}
\newcommand{\uv}{{\mathbf{u}}}
\newcommand{\vv}{{\mathbf{v}}}
\newcommand{\Uv}{{\mathbf{U}}}
\newcommand{\Fv}{{\mathbf{F}}}
\newcommand{\Av}{{\mathbf{A}}}
\newcommand{\Bv}{{\mathbf{B}}}
\newcommand{\fv}{{\mathbf{f}}}
\newcommand{\ev}{{\mathbf{e}}}
\newcommand{\qv}{{\mathbf{q}}}
\newcommand{\tv}{{\mathbf{t}}}
\newcommand{\Ev}{{\mathbf{E}}}
\newcommand{\Yv}{{\mathbf{Y}}}
\newcommand{\Vv}{{\mathbf{V}}}
\newcommand{\Dv}{{\mathbf{D}}}
\newcommand{\ptv}{\boldsymbol{\partial}}
\def \Qv{{\mathbf{Q}}}

\begin{document}

\title{Efficient high order algorithms for fractional integrals and fractional differential equations}
\author{L. Banjai  \thanks{The Maxwell Institute for Mathematical Sciences, School of Mathematical \& Computer Sciences, Heriot-Watt University, Edinburgh EH14 4AS, UK. ({\tt l.banjai@hw.ac.uk})} \and M. L\'opez-Fern\'andez\thanks{Dipartimento di Matematica Guido Castelnuovo, Sapienza Universit\`a di Roma, Piazzale Aldo Moro 5, 00185 Roma, Italy ({\tt lopez@mat.uniroma1.it}) The work of the second author was partially supported by INdAM-GNCS, the Spanish grant MTM2016-75465, and by the Ram\'on y Cajal program of the Ministerio de Economia y Competitividad, Spain.}}
\maketitle

\begin{abstract}
We propose an efficient algorithm for the approximation of fractional integrals by using Runge--Kutta based convolution quadrature. The algorithm is based on a novel integral representation of the convolution weights and a special quadrature for it. The resulting method is easy to implement, allows for high order, relies on rigorous error estimates  and its performance in terms of memory and computational cost is among the best to date. Several numerical results illustrate the method and we describe how to apply the new algorithm to solve fractional diffusion equations. For a class of fractional diffusion equations we give the error analysis of the full space-time discretization obtained by coupling the FEM method in space with Runge--Kutta based convolution quadrature in time.
\end{abstract}

{\bf Keywords:} fractional integral, fractional differential equations, convolution quadrature, fast and oblivious algorithms.

{\bf AMS subject classifications:} 65R20, 65L06, 65M15,26A33,35R11.

\section{Introduction}
Fractional Differential Equations (FDEs) have nowadays become very popular for modeling different physical processes, such as anomalous diffusion \cite{Yus} or viscoelasticity \cite{Suecos_2008,Kar16}. In the present paper we develop a fast and memory efficient method to compute the fractional integral
\begin{equation}\label{fracint}
\mathcal{I}^{\alpha}[f](t)=\frac{1}{\Gamma(\alpha)}\int_0^t (t-s)^{\alpha-1} f(s)\,ds,
\end{equation}
for a given $\alpha \in (0,1)$. A standard discretization of \eqref{fracint} is obtained by convolution quadrature (CQ) based on a Runge-Kutta scheme \cite{LubOst,BanLu_2011}
\begin{equation}
  \label{eq:cq}
\mathcal{I}^{\alpha}[f](t_n) \approx \sum_{j=0}^{n} \w_{n-j} {\bf f}_j,
\end{equation}
where the convolution weights $\w_n$ can be expressed as, see Lemma~\ref{lem:int_cqWw},
\begin{equation}\label{eq:cqw1}
\w_n=\frac{h\sin (\pi \alpha)}{\pi} \int_0^\infty x^{-\alpha} {\bf e}_n(-\dt x)\,dx,
\end{equation}
with ${\bf e}_n(\cdot)$ a function that depends on the Runge-Kutta scheme. For discretizations based on linear multistep methods, see \cite{Lub_frac}.

To compute up to time $T = N\dt$ using formula \eqref{eq:cq} requires $O(N)$ memory and $O(N^2)$ arithmetic operations. Algorithms based on FFT can reduce the computational complexity to $O(N \log N)$ \cite{Lu_88II} or $O(N \log^2 N)$ \cite{hals}, but not the memory requirements; for an overview of FFT algorithms see \cite{lb_ms}. Here we develop algorithms that reduce the memory requirement to $O(|\log \veps|\log N)$ and the computational cost to $O(|\log \veps| N \log N)$, with $\veps$ the accuracy in the computation of the convolution weights.
 Hence, our algorithm has the same complexity as the fast and oblivious quadratures of \cite{LuScha} and \cite{SchaLoLu}, but as we will see, a simpler construction.

The algorithms will depend on an efficient quadrature of \eqref{eq:cqw1} for $n \geq n_0$, with a very moderate threshold value for $n_0$, say $n_0=5$. As ${\bf e}_n(z) = r(z)^n{\bf q}(z)$ and $r(z) = e^{z}+O(z^{p+1})$, where $p$ is the order of the underlying RK method, this is intimately related to the construction of an efficient quadrature for the integral representation of the convolution kernel
\begin{equation}\label{intkernel}
t^{\alpha - 1} = \frac{1}{\Gamma(1-\alpha)} \int_0^\infty x^{-\alpha} e^{-tx} dx,
\end{equation}
with $t \in [n_0\dt, T]$. Note that as $\Gamma(1-\alpha)\Gamma(\alpha) = \pi/\sin(\pi \alpha)$, $h^{-1}\w_n$ is an approximation of $\tfrac1{\Gamma(\alpha)} t^{\alpha-1}$, i.e., the kernel of \eqref{fracint}.

Even though we eventually only require the quadrature for \eqref{eq:cqw1}, we begin with developing a quadrature formula for \eqref{intkernel} for a number of reasons: the calculation for \eqref{intkernel} is cleaner and easier to follow, such a quadrature allows for efficient algorithms that are not based on CQ, and finally once this is available the analysis for \eqref{eq:cqw1} is much shorter. The quadrature we develop  for \eqref{intkernel} is closely related to the one developed in \cite{Li}, the main difference being our treatment of the singularity at $x=0$ by Gauss-Jacobi quadrature and the restriction of $t$ to the finite interval rather than semi-infinite as used in \cite{Li}. Both these decisions allow us to substantially reduce constants in the above asymptotic estimates of memory and computational  costs. Recent references \cite{Bu17,JiangZhang,Ba18} also consider fast computation of \eqref{fracint}, but do not address the approximation of the convolution quadrature approximation exploiting \eqref{eq:cqw1}. Our main contribution here is the development of an efficient quadrature  to approximate \eqref{eq:cqw1} and its use in a fast and memory efficient scheme for computing the discrete convolution \eqref{eq:cq}.

The stability and convergence properties of RK convolution quadrature are well understood, see  \cite{LubOst,BanLM}.  This allows us to apply convolution quadrature not only to the evaluation of fractional integrals, but also to the solution of fractional subdiffusion or diffusion-wave equations of the form
\[
  \partial_t^{\beta} u -\Delta u = f, \qquad u^{(j)}(0) = 0,\ j = 0,\dots, m-1,
\]
with $\beta \in (0,2)$. Here,  $\partial_t^{\beta} = \mathcal{I}^{m-\beta}\partial_t^m$, with $m = \lceil\beta\rceil$, denotes the Caputo fractional derivative. Solutions of such equations typically have low regularity at $t = 0$,  but a discussion of adaptive or modified quadratures for this case is beyond the scope of the current paper. For a careful analysis of BDF2 based convolution quadrature of fractional differential equations see \cite{CuLuPa}.

To our knowledge, underlying high order solvers for ODEs have been considered for the approximation of \eqref{fracint} only at experimental level in \cite{Ba18,BaHe17} and in \cite{SchaLoLu}, where a fast and oblivious implementation of RK based CQ is considered for more general applications than \eqref{fracint}. The fast and oblivious quadratures of \cite{LuScha} and \cite{SchaLoLu} have the same asymptotic complexity as our algorithm, but have a more complicated memory management structure and require the optimization of the shape of the integration contour. Our new algorithm has the advantage of being much easier to implement, as it does not require sophisticated  memory management and the optimization of quadrature parameters is much simpler, and furthermore only real arithmetic is required. The new method is also much better suited for the extension to variable steps  --- this will be investigated in a follow up work. On the other hand, the present algorithm is specially tailored to the application to \eqref{fracint} and related FDEs, whereas the algorithms in \cite{LuScha,SchaLoLu} allow for a wider range of applications.

The paper is organized as follows. In Section 2 we develop and fully analyze a special quadrature for \eqref{fracint}, which uses the same nodes and weights for every $t\in [n_0\dt,T]$. In Section 3, we recall Convolution Quadrature based on Runge--Kutta methods and derive the special representation of the associated weights already stated in \eqref{eq:cqw1}. In Section 4 we derive a special quadrature for \eqref{eq:cqw1}, which uses the same nodes and weights for every $n\in [n_0,N]$, with $T=\dt N$. In Section 5 we explain how to turn our quadrature for the CQ weights into a fast and memory saving algorithm. In Section 6 we test our algorithm with a scalar problem and in Section 7 we consider the application to a fractional diffusion equation.  We provide a complete error analysis of the discretization in space and time of a class of fractional diffusion equations.

\section{Efficient quadrature for $t^{\alpha-1}$}\label{sec:quadt_alpha}

In the following we fix an integer $n_0 > 0$, time step $\dt > 0$, and the final computational time $T > 0$. Throughout, the parameter $\alpha$ is restricted to the interval $(0,1)$. We develop an efficient quadrature for \eqref{intkernel} accurate for $t \in [n_0 \dt, T]$.

\subsection{Truncation}\label{sec:truncexp}
First of all we truncate the integral
\[
t^{\alpha - 1} = \frac{1}{\Gamma(1-\alpha)} \int_0^L x^{-\alpha} e^{-tx} dx
+ \tau(L),
\]
where $\tau(L)$ denotes the truncation error.

\begin{lemma}\label{lem:truncerr}
For $t \geq n_0 \dt$ and $L = A/\dt$ we have that
\begin{equation}
  \label{eq:remainder_int}
|\tau(L)| \le \frac{\dt^{\alpha-1}}{\Gamma(1-\alpha)}  \int_A^\infty x^{-\alpha}e^{-n_0 x} dx.
\end{equation}
\end{lemma}

\begin{proof}
\[
\begin{split}
|\tau(L)| &=  \frac{\dt^{\alpha-1}}{\Gamma(1-\alpha)} \int_A^\infty x^{-\alpha}e^{-\frac{t}{\dt}x}dx\\
&\leq \frac{\dt^{\alpha-1}A^{-\alpha}}{\Gamma(1-\alpha)} \int_A^\infty e^{-n_0 x}dx\\
&= \frac{A^{-\alpha}\dt^{\alpha-1}e^{-n_0 A}}{n_0\Gamma(1-\alpha)}.
\end{split}
\]
\end{proof}

\begin{remark}
Given a tolerance $\tol > 0$, $|\tau(L)| \leq \tol$ if
\begin{equation}\label{truncpara}
  A+\frac{\alpha}{n_0}\log(A)  \ge  \frac1{n_0}\log\left(\frac {\dt^{\alpha-1}}{n_0\Gamma(1-\alpha) \tol} \right).
\end{equation}
Assuming $A\ge 1$ we can choose
$$
A = \log\left(\frac {1}{n_0\Gamma(1-\alpha) \tol} \right) + (1-\alpha)\log \left(\frac{1}{\dt} \right).
$$
However, in practice it is advantageous to use the bound \eqref{eq:remainder_int} to  numerically find the optimal $A$.
\end{remark}

\subsection{Gauss-Jacobi quadrature for the initial interval}\label{sec:gjexp}
We choose an initial integration interval
\[
I_0 = \frac{1}{\Gamma(1-\alpha)} \int_0^{L_0} x^{-\alpha} e^{-tx} dx,
\]
along which we will perform Gauss-Jacobi integration.

Recall the Bernstein ellipse $\mathcal{E}_{\vrho}$, which is given as the image of the circle of radius $\varrho > 1$ under the map $z \mapsto (z+z^{-1})/2$. The largest imaginary part on $\mathcal{E}_{\rho}$ is $(\varrho-\varrho^{-1})/2$ and the largest real part is $(\varrho+\varrho^{-1})/2$.

\begin{theorem}\label{th:gwerr}
  Let $f$ be analytic inside the Bernstein ellipse $\mathcal{E}_{\vrho}$ with $\vrho > 1$ and bounded there by $M$. Then the error of Gauss quadrature with weight $w(x)$ is bounded by
\[
|If-I_Q f| \leq 4M\frac{\vrho^{-2Q+1}}{\vrho-1}\int_{-1}^1 w(x)dx,
\]
where $If = \int_{-1}^1 w(x) f(x) dx$ and $I_Q f = \sum_{j = 1}^Q w_j f(x_j)$ is the corresponding Gauss formula, with weights $w_j > 0$.
\end{theorem}
\begin{proof}
  A proof of this result for $w(x) \equiv 1$ can be found in \cite[Chapter 19]{Tre}. The same proof works for the weighted Gauss quadrature as well. We give the details next.

First of all note that we can expand $f$ in Chebyshev series
\[
f(x) = \sum_{k = 0}^\infty a_k T_k(x)
\]
with $|a_k| \leq 2M\vrho^{-k}$ \cite[Theorem 8.1]{Tre}. If we denote by $f_K(x) = \sum_{k = 0}^K a_k T_k(x)$ the truncated series then
\[
|f-f_K| \leq \frac{2M\vrho^{-K}}{\vrho-1}.
\]
As $I_Q$ is exact for polynomials of degree $2Q-1$, we have that
\[
\begin{split}
|If-I_Qf| &= \left| I(f-f_{2Q-1})-I_Q(f-f_{2Q-1}) \right| \\
&\leq \frac{2M\vrho^{-2Q+1}}{\vrho-1} \left(\int_{-1}^1 w(x) dx + \sum_{j = 1}^{2Q-1} w_j \right)\\
&= \frac{4M\vrho^{-2Q+1}}{\vrho-1}\int_{-1}^1 w(x) dx,
\end{split}
\]
where we have used the fact the weights are positive and integrate constants exactly.
\end{proof}

Changing variables  to the reference interval $[-1,1]$ we obtain
\[
I_0= \frac1{\Gamma(1-\alpha)}\left(\frac{L_0}{2}\right)^{1-\alpha} \int_{-1}^1 e^{-t(y+1)L_0/2}(y+1)^{-\alpha} dy.
\]
We apply Theorem~\ref{th:gwerr} to the case
\begin{equation}\label{fw}
f_0(x)=  \frac1{\Gamma(1-\alpha)}\left(\frac{L_0}{2}\right)^{1-\alpha}e^{-t(x+1)L_0/2}, \qquad w(x)=(1+x)^{-\alpha}
\end{equation}
and denote
$$
\tau_{\text{GJ}}(Q) := If_0-I_Qf_0 = \int_{-1}^1 f_0(x)w(x)\,dx - \sum_{j=1}^{Q}w_j f_0(x_j).
$$

\begin{theorem}\label{th:gjerr}
For $t \in [0,T]$ and any $Q \geq 1$ we have the bound
\[
|\tau_{\mathrm{GJ}}(Q)| \leq \frac{4L_0^{1-\alpha}}{\Gamma(2-\alpha)} \left(1+ \frac{TL_0}{4Q}\right) \left( \frac{eTL_0} {8Q}\right)^{2Q}.
\]
\end{theorem}
\begin{proof}
Since $f_0$ in \eqref{fw} is an entire function, by Theorem~\ref{th:gwerr} we can estimate
\begin{eqnarray*}
|\tau_{\text{GJ}}(Q)| &\le&  \frac4{\Gamma(1-\alpha)}\left(\frac{L_0}{2}\right)^{1-\alpha} \left(\int_{-1}^1 (1+x)^{-\alpha} \,dx \right)\, \min_{\varrho > 1} \left( \frac{\varrho^{-2Q+1}}{\rho-1} \max_{\xi \in \mathcal{E}_{\rho}} \left|e^{-t(\xi+1)L_0/2} \right| \right)\\
&= & \frac4{\Gamma(1-\alpha)}\frac{L_0^{1-\alpha}}{1-\alpha} \min_{\varrho > 1} \left( \frac{\varrho^{-2Q+1}}{\rho-1} \max_{\xi \in \mathcal{E}_{\rho}} e^{-t(\Re\xi+1)L_0/2} \right)\\
&=& \frac{4L_0^{1-\alpha}}{\Gamma(2-\alpha)} \min_{\varrho > 1} \left( \frac{\varrho^{-2Q+1}}{\rho-1} e^{t(\vrho+\vrho^{-1}-2)L_0/4} \right)\\
&\leq&  \frac{4L_0^{1-\alpha}}{\Gamma(2-\alpha)} \min_{\varrho > 1} \left( \frac{\varrho^{-2Q+1}}{\rho-1} e^{T(\vrho+\vrho^{-1}-2)L_0/4} \right).
\end{eqnarray*}


Let $\varrho = e^{\delta}$ with $\delta > 0$. Then the error bound can be written as
\[
|\tau_{\text{GJ}}(Q)| \le \frac{4L_0^{1-\alpha}}{\Gamma(2-\alpha)} \min_{\delta > 0} \frac{e^{\delta}}{e^{\delta}-1} e^{-2Q\delta + L_0 T \left(\cosh \delta -1 \right)/2}.
\]
We now choose $\delta$ so that it maximises the function
\[
g(\delta) = 2Q\delta - L_0T\left(\cosh \delta -1 \right)/2.
\]
As
\[
g'(\delta) = 2Q-L_0T\sinh \delta/2, \quad g''(\delta) = -L_0T\cosh \delta /2 < 0,
\]
we have a maximum at
\[
2Q-L_0 T \sinh \delta/2= 0 \implies \delta = \sinh^{-1}\left(\frac{4Q}{TL_0}\right).
\]

Using the identities
$$
\sinh^{-1}y = \log\left( y+\sqrt{1+y^2} \right), \quad \cosh x = \sqrt{1+\sinh^2 x},
$$
we derive an error estimate with the above choice of $\delta$:
\[
\begin{split}
|\tau_{\text{GJ}}(Q)| &\le  \frac{4L_0^{1-\alpha}}{\Gamma(2-\alpha)} \left(1+ \frac{TL_0}{4Q} \right) e^{-2Q\delta  +L_0T(\cosh \delta-1)/2} \\
& \le \frac{4L_0^{1-\alpha}}{\Gamma(2-\alpha)} \left(1+ \frac{TL_0}{4Q}\right) \left( \frac{TL_0} {8Q}\right)^{2Q} e^{L_0T \left(-1+\sqrt{1+(4Q/(TL_0))^2}\right)/2}\\
&\leq \frac{4L_0^{1-\alpha}}{\Gamma(2-\alpha)} \left(1+ \frac{TL_0}{4Q}\right) \left( \frac{TL_0} {8Q}\right)^{2Q} e^{2Q},
\end{split}
\]
where in the last  step above we have used that $-1+\sqrt{1+x^2} \leq x$ for $x > 0$.  This gives the stated result.
\end{proof}

\subsection{Gauss quadrature on increasing intervals}
We next split the remaining integral as
$$
\frac1{\Gamma(1-\alpha)}\int_{L_{0}}^{L} x^{-\alpha}e^{-xt} \,dx = \sum_{j = 1}^J I_j,
$$
where
\[
\begin{split}
I_j &=\frac1{\Gamma(1-\alpha)}\int_{L_{j-1}}^{L_{j}} x^{-\alpha}e^{-xt} \,dx \\
&= \frac{\DL_j}{2\Gamma(1-\alpha)}e^{-L_{j-1} t} \int_{-1}^1 \left( L_{j-1} + \frac{\DL_j}{2} (y+1) \right)^{-\alpha} e^{-t(y+1)\DL_j/2} \,dy,
\end{split}
\]
where $\DL_j = L_j-L_{j-1}$, $j=1,\dots,J$, with  $L_J = L$. The intervals are chosen so that for some $B \geq 1$, $\DL_j = BL_{j-1}$, i.e., $L_j = (B+1)L_{j-1}$ and $J= \lceil \log_{B+1} L/L_0\rceil$. To each integral we apply standard, i.e., $w(x) \equiv 1$ in Theorem~\ref{th:gwerr}, Gauss quadrature with $Q$ nodes and denote the corresponding error by
\[
\tau_j(Q) := If_j-I_Qf_j
\]
with
\begin{equation}
  \label{eq:f2}
f_j(x) = \frac{4\DL_j}{\Gamma(1-\alpha)}e^{-L_{j-1} t} \left( L_{j-1} + \frac{\DL_j}{2} (x+1) \right)^{-\alpha} e^{-t(x+1)\DL_j/2}.
\end{equation}

\begin{theorem}\label{th:gerr}
For any $Q \geq 1$ and $t \geq 0$
\[
|\tau_j(Q)| \le \frac{4B L^{1-\alpha}_{j-1}}{\Gamma(1-\alpha)}\min_{0<\veps<1} \frac{g(\veps,B)^{-2Q+1}}{g(\veps,B)-1} \veps^{-\alpha}e^{-tL_{j-1}\veps},
\]
with
\[
g(\veps,B) = 1+\frac 2B (1-\veps) + \sqrt{ \left(1+\frac 2B (1-\veps)\right)^2 -1}.
\]
\end{theorem}
\begin{proof}
Note that the integrand $f_j$ in \eqref{eq:f2} is now not entire and
there will be a restriction $\vrho < \vrho_{\max}$  on the choice of the Bernstein ellipse $\mathcal{E}_{\vrho}$ in order to avoid the singularity of the fractional power. In particular we require
$$
L_{j-1}-\frac{\Delta L_j}{4}(\vrho+\vrho^{-1}-2)
=L_{j-1}\left(1-\frac{B}{4}(\vrho+\vrho^{-1}-2)\right)
> 0,
$$
which is satisfied for $1<\vrho <\vrho_{\max}$ and
$$
\vrho_{\max}=1+\frac2{B}(1+\sqrt{1+B}).
$$
Setting
$$
\veps(\vrho) = 1-\frac{B}{4}(\vrho+\vrho^{-1}-2)
$$
we see that $\veps \in (0,1)$ for $\vrho \in (1, \vrho_{\max})$ and that
\[
  L_{j-1}-\frac{\DL_j}{4}(\vrho+\vrho^{-1}-2) = L_{j-1}\veps.
  \]
  Hence
  \[
\begin{split}
|\tau_j(Q)| &\le \frac{4\DL_{j}}{\Gamma(1-\alpha)}L^{-\alpha}_{j-1} \min_{1<\vrho<\vrho_{\max}}  \frac{\vrho^{-2Q+1}}{\vrho-1}  \veps^{-\alpha}e^{-tL_{j-1}\veps}.
\end{split}
\]
The result is now obtained by using
$$
\cosh^{-1} y = \log(y+\sqrt{y^2-1}), \qquad y\ge 1,
$$
to show that $\vrho = g(\veps,B)$.
\end{proof}

\begin{remark}
Choosing for instance $\veps=0.1$ and $B=1$ in the above estimate, we obtain
$$
\vrho_{\max}=g(0,B)=3+2\sqrt{2}=5.83
$$
and
$$
|\tau_j(Q)| \le \frac{10^{\alpha}  L^{1-\alpha}_{j-1}\exp(-0.1t L_{j-1})}{1.1\Gamma(1-\alpha)} (5.41)^{-2Q+1}.
$$
As we will require a uniform bound for $t \in [t_{n_0},T]$, we can substitute $t = t_{n_0}$ in this estimate.
\end{remark}

\section{Runge--Kutta Convolution Quadrature}\label{subsec:rkcq}
Let us consider an $s$-stage Runge-Kutta method described by the coefficient matrix $\A =
(a_{ij})_{i,j=1}^s \in \bR^{s\times s}$, the vectors of weights $\bv =
(b_1,\ldots,b_s)^T \in \bR^s$ and the vector of abcissae $\cv = (c_1,\ldots,c_s)^T \in
[0,1]^s$. We assume that the method is $A$-stable, has classical order $p\ge 1$, stage order $q$ and satisfies $a_{s,j}=b_j$, $j=1,\dots,s$, \cite{HaWII}.
The corresponding stability function is given by
\begin{equation}\label{stabfun}
r(z) = 1+z\bv^T (\mathbf{I}-z \A)^{-1} \bone,
\end{equation}
where
\[
\bone = (1,1,\dots,1)^T.
\]
Our assumptions imply the following properties:
\begin{enumerate}
\item $c_s=1$.
    \item $r(\infty) = \bv^T \A^{-1}\bone-1 = 0$.
  \item $r(z) = e^z+O(z^{p+1})$
    \item $|r(z)| \leq 1$ for $\Re z \leq 0$.
  \end{enumerate}
Important examples of RK methods satisfying our assumptions are Radau
IIA and Lobatto IIIC methods.

Let us consider the convolution
\begin{equation}
  \label{eq:conv}
K(\partial_t) f := \int_0^t k(t-\tau) f(\tau) d\tau,
\end{equation}
where $K(z)$ denotes the Laplace transform of the convolution kernel $k(t)$. $K$ is assumed to be analytic for $\Re z > 0$ and bounded there as $|K(z)| \leq |z|^{-\mu}$ for some $\mu > 0$. The operational notation $K(\partial_t) f$  introduced in \cite{Lu_88I}, is useful in emphasising certain properties of convolutions. Of particular importance is the composition rule,  namely, if
$K(s) = K_1(s)K_s(s)$ then $K(\partial_t)f = K_1(\partial_t) K_2(\partial_t) f$.. This will be used when solving fractional differential equations in  Section~\ref{subsec:algFDE}.

If $\mu < 0$, the convolution is defined by
  \[
    K(\partial_t)f = \left(\frac{d}{dt}\right)^m K_m(\partial_t)f,
  \]
  where $K_m(z) = z^{-m} K(z)$ and $m$ smallest integer such that $m > -\mu$.

For $K(z) = z^{-\alpha}$ the convolution coincides with the fractional integral of order $\alpha$, i.e., according to the operational notation, we can write
\[
  \mathcal{I}^{\alpha}[f](t) = \partial_t^{-\alpha} f(t), \qquad t > 0, \;\alpha \in (0,1).
\]
For $\beta > 0$, $\partial_t^\beta$ is equivalent to the Riemann-Liouville fractional derivative,  see definition \eqref{eq:frac_der}.

Runge--Kutta convolution quadrature has been derived in \cite{LubOst} and applied to \eqref{eq:conv} provides approximations at time-vectors ${\bf t}_{n}=(t_{n,j})_{j=1}^s$, with $t_{n,j}=t_{n} + c_j \dt$ and $t_n=n\dt$, defined by
\begin{equation}
  \label{eq:rkcq_gen}
  K(\partial_t)f(\tv_{n})  \approx K(\ptv_t^{h})f(\tv_{n})
:=     \sum_{j=0}^{n} \Wv_{n-j}(K)\fv_j,
\end{equation}
where $(K(\partial_t)f(\tv_{n}))_{\ell} = K(\partial_t)f(t_{n,\ell})$, $(\fv_j)_\ell = f(t_{j,\ell})$ and the weight matrices $\Wv_j$ are the coefficients of the power series
\begin{equation}\label{powomega}
K\left( \frac{\Deltav(\zeta)}{h} \right) = \sum_{j=0}^{\infty} \Wv_j(K) \zeta^j
\end{equation}
with
\begin{equation}
  \label{eq:rksymbol}
\Deltav(\zeta)= \Bigl(\A + {\zeta \over 1-\zeta}\bone \bv^T\Bigr)^{-1}=
\A^{-1}-\zeta \A^{-1}\bone \bv^T \A^{-1}.
\end{equation}
The notation in \eqref{eq:rkcq_gen} again emphasises that the composition rule holds also after discretization:  if
$K(s) = K_1(s)K_s(s)$ then $K(\ptv_t^{h})f = K_1(\ptv_t^{h}) K_2(\ptv_t^{h}) f$.

The last row in \eqref{eq:rkcq_gen} defines the approximation at the time grid $t_{n+1}$, since $c_s=1$. Denoting $\omegav_j(K)$ the last row of $\Wv_j(K)$, the approximation reads
\begin{equation}\label{RKCQ}
  K(\partial_t)f(t_{n+1})  \approx K(\partial_t^{h})f(t_{n+1})
:= \sum_{j=0}^{n} \omegav_{n-j}(K)\fv_j, \quad (\fv_j)_\ell = f(t_{j,\ell}).
\end{equation}
For the rest of the paper we will denote by $\Wv_j = \Wv_j(K)$ and $\omegav_j = \omegav_j(K)$ the weights for the fractional integral case, i.e., for $K(z) = z^{-\alpha}$.

\begin{remark}[Notation]\label{rem:notation}
  We have defined the discrete convolution $K(\partial_t^h) f$ for functions $f$. For a sequence $\mathbf{f}_0, \dots, \mathbf{f}_N \in \mathbb{R}^s$, we use the same notation $K(\partial_t^h)\mathbf{f}$ to denote
  \[
    K(\partial_t^h)\mathbf{f}(t_{n+1}) = \sum_{j=0}^{n} \omegav_{n-j}(K)\mathbf{f}_j,\qquad  n = 0, \dots, N,
  \]
  and similarly for $K(\ptv_t^h)\mathbf{f}(\mathbf{t}_n)$ with the meaning
  \[
        K(\ptv_t^h)\mathbf{f}(\tv_n) = \sum_{j=0}^{n} \Wv_{n-j}(K)\mathbf{f}_j
      \]
      and
        \[
        K(\ptv_t^h)\mathbf{f}(t_{n,\ell}) = \left(\sum_{j=0}^{n} \Wv_{n-j}(K)\mathbf{f}_j\right)_\ell.
      \]
      Note also that
      \[
        K(\ptv_t^h)\mathbf{f}(t_{n,s}) =  K(\partial_t^h)\mathbf{f}(t_{n+1}).
        \]
\end{remark}

FFT techniques based on \eqref{powomega} can be applied to compute at once all the required $\Wv_j$, $j=0,\dots,N$, with $N=\lceil T/\dt \rceil$, \cite{Lu_88II}. The computational cost associated to this method is $O(N\log(N))$. It implies precomputing and keeping in memory all weight matrices for the approximation of every $\mathcal{I}^{\alpha}[f](t_n)$, $n=1,\dots,N$, see \cite{Ban10} for details and many experiments.

The following error estimate for the approximation of \eqref{fracint} by \eqref{RKCQ} is given by \cite[Theorem 2.2]{LubOst}. Notice that we allow $K(z)$ to be a map between two Banach spaces with appropriate norms denoted by $\|\cdot\|$ in the following. This will be needed in Section~\ref{sec:FDE}.

\begin{theorem}\label{th:LuOs}
Assume that there exist $c\in \bR$, $0<\delta<\frac{\pi}{2}$ and $M>0$ such that $K(z)$ is analytic in a sector $|\arg(z-c)|<\pi-\delta$ and satisfies there the bound $\|K(z)\| \le M|z|^{-\alpha}$. Then if $f\in C^{p}[0,T]$, there exists $\dt_0>0$ and $C>0$ such that for $\dt\le \dt_0$ it holds
\begin{align*}
\left\| K(\partial_t)f(t_n) -
K(\partial_t^{\dt})f(t_{n}) \right\|
&\le C h^p  \sum_{\ell=0}^q \left(1+t_n^{\alpha+\ell-p} \right)  \|f^{(\ell)}(0)\|
\\ & + C \left( h^p + h^{q+1+\alpha} |\log(h)|\right) \left( \sum_{\ell=q+1}^{p-1} \|f^{(\ell)}(0)\| + \max_{0\le \tau\le t_n} \|f^{(p)}(\tau)\| \right).
\end{align*}

\end{theorem}

\subsection{Real integral representation of the CQ weights} \label{subsec:intomega}

The convolution quadrature weights $\w_j$ can also be expressed as \cite{SchaLoLu}
\begin{equation}
\w_n=\frac{\dt}{2\pi i} \int_{\Gamma} z^{-\alpha} \ev_n(\dt z)\,dz,
\end{equation}
for $\ev_n(\lambda)$ a function which depends on the ODE method underlying the CQ formula and an integration contour $\Gamma$ which can be chosen as a Hankel contour beginning and ending in the left half of the complex plane.

\begin{lemma}\label{lem:int_cqWw}
The weights are given by
\begin{equation}\label{lem:int_cqW}
\Wv_n=\frac{h\sin (\pi \alpha)}{\pi} \int_0^{\infty} x^{-\alpha} \Ev_n(-\dt x)\,dx,
\end{equation}
  and
\begin{equation}\label{lem:int_cqw}
\w_n= \frac{h\sin (\pi \alpha)}{\pi} \int_0^{\infty} x^{-\alpha} {\bf e}_n(-\dt x)\,dx,
\end{equation}
where
\begin{equation}\label{en}
\left( \Deltav(\zeta) -zI \right)^{-1}= \sum_{n=0}^{\infty}\Ev_n(z)\zeta^n
\end{equation}
and ${\bf e}_n(z)$ is the last row of $\Ev_n(z)$.

Explicit formulas for $\Ev_n$ and ${\bf e}_n$ are given by
\begin{equation}\label{Enrk}
\Ev_0 = \A(I-z\A)^{-1}, \qquad \Ev_n(z)=r(z)^{n-1}(I-z\A)^{-1}\bone \qv(z)
\end{equation}
and
\begin{equation}\label{enrk}
{\bf e}_n(z)=r(z)^n \qv(z),
\end{equation}
where $r$ is the stability function of the method and ${\bf q}(z)={\bf b}^T(I-z\A)^{-1}$.
\end{lemma}
\begin{proof}
  Since $z^{-\alpha}$ is analytic in the whole complex plane but for the branch cut on the negative real axis, the Hankel contour $\Gamma$ can  be degenerated into negative real axis as in the derivation of the real inversion formula for the Laplace transform \cite[Section 10.7]{Henrici_II} to obtain
  \[
    \begin{split}
      \w_n &= \frac{\dt}{2\pi i} \int_0^{\infty} (e^{i\pi \alpha}-e^{-i\pi \alpha})x^{-\alpha} \ev_n(-\dt x)\,dx\\
      &= \frac{\dt \sin(\pi \alpha)}{\pi} \int_0^{\infty} x^{-\alpha} \ev_n(-\dt x)\,dx.
          \end{split}
  \]
The expression for $\Wv_n$ is obtained in the same way and the explicit formulas for $\Ev_n$ and ${\bf e}_n$ can be found in \cite{SchaLoLu}.
\end{proof}

The next properties will be used later in Section~\ref{sec:quadcqw}
\begin{lemma}\label{lem:rbound}
There exist constants $\gamma > 1$, $b > 0$ and  $C_{\qv} > 0$ such that
\[
|r(z)| \leq e^{\gamma \Re z}, \qquad \text{ for }\ 0 \leq \Re z \leq b,
\]
and
\[
\|{\bf q}(z)\| \leq C_{\qv}, \qquad \text{ for }\ \Re z \leq b,
 \]
where $C_{\qv}$ depends on the choice of the norm $\|\cdot \|$.
\end{lemma}
\begin{proof}
Fix a $b > 0$ such that all the poles of $r(z)$ (and hence ${\bf q}(z)$) belong to $\Re z > b$. Define now
\begin{equation}
  \label{eq:gamma_opt}
\gamma = \sup_{0 \leq \Re z \leq b}  \frac{1}{\Re z} \log |r(z)|
= \max\left\{1,\frac{1}{b} \sup_{\Re z = b} \log |r(z)|\right\},
\end{equation}
where we have used the properties of $r(z)$ to see that $\sup _{\Re z = 0}  \frac{1}{\Re z}\log |r(z)| = 1$.

Recall that $\qv(z) = \bv^T(I-z\A)^{-1}$. As all the singularities of $\qv$ are in the half-plane $\Re z > b$ and $\|\qv(z)\| \rightarrow 0$ as $|z| \rightarrow \infty$, we have that $\|\qv(z)\|$ is bounded in the region $\Re z \leq b$.
\end{proof}

\begin{remark}
  \begin{enumerate}[(a)]
  \item   Note that for BDF1 we can choose $b \in (0,1)$. Hence, $\gamma = b^{-1}\log\frac1{1-b}$ and since $\qv(z) = r(z)$ for BDF1, we can set $C_{\qv} =  e^{\gamma b}$.
\item For the 2-stage Radau IIA method we have
\[
r(z) = \frac{2z+6}{z^2-4z+6}, \quad \qv(z) = \frac{1}{2(z^2-4z+6)}
\begin{bmatrix}
  9 & 3-2z
\end{bmatrix}.
\]
As the  poles of $r$ and $\qv$ are at $z = 2\pm \sqrt{2}\mi$, we can choose any $ b\in (0,2)$ and obtain the optimal $\gamma$ numerically using \eqref{eq:gamma_opt}. For example for $b = 1$, we can choose $\gamma \approx 1.0735$. Similarly we can compute $C_{\qv}$ by computing
\[
C_{\qv} = \sup_{\Re z = 0 \text{ or } \Re z = b} \|\qv(z)\|.
\]
For $b = 1$ and the Euclidian norm we have $C_{\qv} \approx 1.6429$.
Using the same procedure, for $b = 3/2$, we have $\gamma \approx 1.2617$ and $C_{\qv} \approx 3.3183$.
\item For the 3-stage Radau IIA method
  the poles of $r(z)$ and $\qv(z)$ belong to $\Re z \geq \frac{9^{2/3}}{6}-\frac{9^{1/3}}2+3 \approx 2.681$. Choosing $b = 1$ gives $\gamma \approx 1.0117$ and $C_{\qv} \approx 1.1803$, whereas for $b = 1.5$ we obtain $\gamma \approx 1.0521$ and $C_\qv \approx 1.7954$.
  \end{enumerate}

\end{remark}

\begin{lemma}\label{lem:ebound}
  There exist  constants $c > 0$ and $x_0 > 0$ such that
\[
\max\{\|\ev_n(z)\|, \|\Ev_n(z)\|\} \leq |x_0-c\Re z|^{-n-1}, \qquad \text{for } \Re z <  0.
\]
\end{lemma}
\begin{proof}
  Using that $r(\infty) = \bv^T \A^{-1}\bone-1 = 0$ it can be shown that
  $r(z) = \bv^T \A^{-1}(I-z \A)^{-1}\bone$.  Let all eigenvalues of $\A^{-1}$ and hence all poles of $r(z)$, $\qv(z)$, and $(I-z\A)^{-1}$ lie in $\Re z \geq \tilde x_0 > 0$ . There exists a constant $C$ such that for all  $\Re z < 0$
  \[
    \max\{|r(z)|,\|\qv(z)\|, \|(I-z\A)^{-1}\| \} \leq C|\Re z-\tilde x_0|^{-1}
    \leq |x_0-c\Re z|^{-1},
\]
where we can set $x_0 = \frac{1}{C}\tilde   x_0$ and $c = \frac1C$.
\end{proof}

\begin{remark}
  \begin{enumerate}
  \item
    For BDF1, $c = 1$ and $x_0 = 1$.
\item  For 2-stage Radau IIA the constant can be obtained following the proof. Namely we choose $\tilde x_0 = 2$ and find numerically that
\[
\max\{|r(z)|, \|\qv(z)\|, \|(I-z\A)^{-1}\|\} |\Re z-\tilde x_0|  \leq 2, \qquad \Re z \leq 0.
\]
Hence we can choose $C = 2$ and $c = 1/2$ and $x_0 = 2/C = 1$.
\item Similarly, for 3-stage Radau IIA we choose $\tilde x_0 = 2.6811$ and find that
  \[
\max\{|r(z)|, \|\qv(z)\|, \|(I-z\A)^{-1}\|\} |\Re z-\tilde x_0|  \leq 3.0821, \qquad \Re z \leq 0.
\]
Hence we can choose $C = 3.0821$ and $c = 1/C = 0.3245$ and $x_0 = \tilde x_0/C = 0.8699$.
  \end{enumerate}
\end{remark}
In the rest of the Section our goal is to derive a good quadrature for the approximation of $\omegav_n$ and $\Wv_n$. We will perform the same steps as in Section 2 for the $\omegav_n$. The same quadrature rules will give essentially the same error estimates for the $\Wv_n$; see Remark~\ref{rem:W_comment}.

\section{Efficient quadrature for the CQ weights}\label{sec:quadcqw}

Analogously to the  the continuous case \eqref{fracint},  we fix $n_0$, $\dt$, and $T$ and develop an efficient quadrature for the CQ weights representation \eqref{lem:int_cqw}, for $n\dt \in [(n_0+1)\dt, T]$  and $\alpha \in (0,1)$.

\subsection{Truncation of the CQ weights integral representation}

Again we truncate the integral
\[
\w_n = \frac{\dt \sin(\pi \alpha)}{\pi} \int_0^{L} x^{-\alpha} \ev_n(-\dt x)\,dx
+ \tauv(L)
\]
and give a bound on the truncation error $\tauv(L)$.
\begin{lemma}
  With the choice $L = A\dt^{-1}$, the truncation error is bounded as
  \begin{equation}
    \label{eq:remainder_intw}
\|\tauv(L)\| \leq
 \frac{\dt^{\alpha}\sin(\pi \alpha)}{\pi}\int_A^\infty \|\ev_n(-x)\|x^{-\alpha}  dx
  \end{equation}
or more explicitly
\[
\|\tauv(L)\| \leq  \frac{\dt^{\alpha}\sin(\pi \alpha)}{cn\pi}A^{-\alpha}(x_0+cA)^{-n}.
\]
\end{lemma}
\begin{proof}
From Lemma~\ref{lem:ebound} we have that
\[
\begin{split}
\|\tauv(L)\| &\leq
 \frac{\dt^{\alpha}\sin(\pi \alpha)}{\pi}\int_A^\infty \|\ev_n(-x)\|x^{-\alpha}  dx\\
&\leq \frac{L^{-\alpha}\sin(\pi \alpha)}{\pi}\int_A^\infty (x_0+cx)^{-n-1}  dx\\
&= \frac{\dt^{\alpha}\sin(\pi \alpha)}{cn \pi}A^{-\alpha}(x_0+cA)^{-n}.
\end{split}
\]
\end{proof}

\begin{corollary}\label{coro:L}
  Let $L = A/\dt$. Given $\tol > 0$, choosing
\[
A > \left(\frac{\dt^\alpha\sin(\pi \alpha)}{ \tol \,n\pi c^{n+1}}\right)^{\frac{1}{n+\alpha}}
\]
ensures $\|\tauv(L)\| \leq \tol$. The estimate becomes uniform in $n>n_0$ by setting $n=n_0+1$ in the above error bound.
\end{corollary}
\begin{proof}
  We have from above
\[
\|\tauv(L)\| \leq  \frac{\dt^{\alpha}\sin(\pi \alpha)}{cn\pi}A^{-\alpha}|cA+z_0|^{-n}
\leq \frac{\dt^{\alpha}\sin(\pi \alpha)}{n\pi}A^{-\alpha-n} c^{-n-1},
\]
from which the result follows.
\end{proof}
\begin{remark}
In practice we find that instead of using Corollary~\ref{coro:L}, better results are obtained if a simple numerical search is done to find the optimal $A$ such that the right-hand side in \eqref{eq:remainder_intw} with $n = n_0+1$ is less than $\tol$. To do this, we start from $A=0$ and iteratively approximate the integral in \eqref{eq:remainder_intw} for increased values of $A$ ($A \leftarrow A+0.125$ in our code) until the resulting quantity is below our error tolerance. The approximation of the integrals is done by the MATLAB built-in routine {\tt integral}. Notice that this has to be done only once for each RK-CQ formula and value of $\alpha \in (0,1)$.
\end{remark}

\subsection{Gauss-Jacobi quadrature for the CQ weights}\label{sec:gjw}
In a similar way as in Section~\ref{sec:gjexp}, we consider the approximation of the integral
\[
\w_n = \frac{\dt\sin(\pi \alpha)}{\pi} \int_0^{\infty} x^{-\alpha} \ev_n(-x\dt)\,dx
\]
and investigate in the first place the approximation of
$$
\Iv_{0,n}:=\frac{\dt\sin(\pi \alpha)}{\pi} \int_0^{L_0} x^{-\alpha} \ev_n(-x\dt)\,dx,
$$
for some suitable $L_0>0$ by using Gauss-Jacobi quadrature. Changing variables as in Section~\ref{sec:gjexp} we obtain
\[
\Iv_{0,n}=\frac{\dt \sin(\pi \alpha)}{\pi}\left(\frac{L_0}{2}\right)^{1-\alpha} \int_{-1}^1 (y+1)^{-\alpha} \ev_n(-\dt (y+1)L_0/2)dy
\]
and apply Theorem~\ref{th:gwerr} to estimate the error
\[
  \tauv_{\mathrm{GJ},n}(Q) = \Iv_Q \fv_0-\Iv \fv_0
\]
with the weight $w(x) = (x+1)^{-\alpha}$  and integrand
\[
  \fv_0(x) = \frac{\dt \sin(\pi \alpha)}{\pi}\left(\frac{L_0}{2}\right)^{1-\alpha}  \ev_n(-\dt (x+1)L_0/2).
\]

\begin{theorem}\label{th:errgj}
  Let
  \[
  \vrho_{\max} = 1+\frac{2b}{L_0 \dt}
  + \sqrt{\left(\frac{2b}{L_0 \dt}\right)^2+\frac{4b}{L_0 \dt}},
\]
with $b$ and $\gamma$ from Lemma~\ref{lem:rbound}, and
\[
\vrho_\opt = \frac{4Q}{\gamma TL_0} + \sqrt{1+\left(\frac{4Q}{\gamma TL_0}\right)^2}.
\]
If $\vrho_\opt \in (1,\vrho_{\max})$, we have the bound
\[
  \|\tauv_{\mathrm{GJ},n}(Q)\| \leq  C_{\qv}\frac{\dt L_0^{1-\alpha}\sin(\pi \alpha)}{\pi(1-\alpha)}\left(1+ \frac{\gamma TL_0}{4Q}\right) \left( \frac{e\gamma TL_0} {8Q}\right)^{2Q}.
\]
Otherwise we have the bound
\[
  \|\tauv_{\mathrm{GJ},n}(Q)\| \leq C_{\qv}\frac{\dt L_0^{1-\alpha}\sin(\pi \alpha)}{\pi(1-\alpha)}
\left( \frac{\vrho_{\max}^{-2Q+1}}{\vrho_{\max}-1}
e^{\gamma T b/\dt} \right).
  \]
\end{theorem}
\begin{proof}
We again consider the Bernstein ellipse $\mathcal{E}_{\vrho}$ around $[-1,1]$, but now in order to be able to use Lemma~\ref{lem:rbound} and avoid the singularities of $\ev_n(z)$ in the right-half plane we have a restriction on $\vrho$. Namely, the maximal value of $\vrho$ is given by
$$
\dt\left(\vrho_{\max}+\vrho^{-1}_{\max}-2 \right) L_0/4= b,
$$
which implies, writing $\vrho_{\max}=e^{\delta_{\max}}$,
$$
\cosh(\delta_{\max})-1 = \frac{2b}{L_0 \dt},
$$
and thus
$$
\delta_{\max} = \cosh^{-1} \left(1+ \frac{2b}{L_0 \dt}\right)
$$
giving the expression for $\vrho_{\max}$ from the statement of the theorem.
The error estimate for Gauss-Jacobi quadrature then reads, by using Lemma~\ref{lem:rbound},
\[
\begin{split}
\left\|\tauv_{\mathrm{GJ},n}\right\| &\le \frac{\dt L_0^{1-\alpha}\sin(\pi \alpha)}{\pi(1-\alpha)} \min_{1<\vrho\leq\vrho_{\max}}
\left( \frac{\vrho^{-2Q+1}}{\vrho-1}
\max_{\zeta \in \mathcal{E}_{\rho}} \left\| \ev_n(-\dt(\zeta+1)L_0/2) \right\| \right)\\
& \le C_{\qv}\frac{\dt L_0^{1-\alpha}\sin(\pi \alpha)}{\pi(1-\alpha)} \min_{1<\vrho\leq\vrho_{\max}}
\left( \frac{\vrho^{-2Q+1}}{\vrho-1}
e^{\gamma t_n L_0 (\vrho +\vrho^{-1}-2)/4} \right) \\
&  \le C_{\qv}\frac{\dt L_0^{1-\alpha}\sin(\pi \alpha)}{\pi(1-\alpha)} \min_{1<\vrho<\vrho_{\max}}
\left( \frac{\vrho^{-2Q+1}}{\vrho-1}
e^{\gamma T L_0 (\vrho +\vrho^{-1}-2)/4} \right).
\end{split}
\]
Proceeding as in Section~\ref{sec:gjexp} with $\gamma T$ in place of $T$ we obtain the bound
\[
\|\tauv_{\mathrm{GJ},n}(Q)\| \leq C_{\qv}\frac{\dt L_0^{1-\alpha}\sin(\pi \alpha)}{\pi(1-\alpha)}\left(1+ \frac{\gamma TL_0}{4Q}\right) \left( \frac{e\gamma TL_0} {8Q}\right)^{2Q},
\]
provided that the optimal value for $\vrho$ is within the accepted interval
$$
\vrho_{opt} = \frac{4Q}{\gamma TL_0} + \sqrt{1+\left(\frac{4Q}{\gamma TL_0}\right)^2} \in (1,\vrho_{\max}),
$$
otherwise we make the choice $\vrho = \vrho_{\max}$.
\end{proof}
\begin{remark}
  In all our numerical experiments, we have found that $\vrho_{\opt} < \vrho_{\max}$.
\end{remark}
\subsection{Gauss quadrature on increasing intervals for the CQ weights}
We next split the remaining integral into the sum
\[
\frac{\dt\sin(\pi \alpha)}{\pi}\int_{L_{0}}^{L} x^{-\alpha} \ev_n(-x\dt) \,dx
= \sum_{j = 1}^J I_{n,j},
\]
where
\[
I_{n,j} =\frac{\dt\sin(\pi \alpha)}{\pi}\int_{L_{j-1}}^{L_{j}} x^{-\alpha} \ev_n(-x\dt) \,dx.
\]
The intervals are again chosen so that for some $B \geq 1$, $L_j = (B+1)L_{j-1}$. To each integral we apply standard Gauss quadrature, i.e., $w(x) \equiv 1$ in Theorem~\ref{th:gwerr}, with $Q$ nodes and denote the corresponding error
by $\tauv_{n,j}(Q)$.

\begin{theorem}\label{th:gerr_wj}
\[
\|\tauv_{n,j}(Q)\| \le \frac{4\dt B L^{1-\alpha}_{j-1}\sin(\pi \alpha)}{\pi}\min_{0<\veps<1} \frac{g(\veps,B)^{-2Q+1}}{g(\veps,B)} \veps^{-\alpha}\min(C_{\qv}, |x_0+cL_{j-1}\dt \veps|^{-n-1}),
\]
with constants $C_{\qv},c,x_0$ from Lemmas~\ref{lem:rbound} and \ref{lem:ebound} and
\begin{equation}\label{fepsB}
g(\veps,B) = 1+\frac 2B (1-\veps) + \sqrt{ \left(1+\frac 2B (1-\veps)\right)^2 -1}.
\end{equation}
\end{theorem}
\begin{proof}
The proof is the same as the proof of Theorem~\ref{th:gerr}, we only need to combine the facts that $|r(z)| \leq 1$ for $\Re z \leq 0$, the bound $\|q(z)\| \leq C_{\qv}$ from Lemma~\ref{lem:rbound}, and the bound from Lemma~\ref{lem:ebound}.
\end{proof}

\begin{remark}
 To obtain a uniform bound for $t_n \in [t_{n_0+1}, T]$, we replace $n$ by $n_0+1$ in the above bound.
\end{remark}

\begin{remark}\label{rem:W_comment}
 We have developed the quadrature for the weights $\w_n$. However, up to a small difference in constants, the same error estimates hold for the matrix weights $\Wv_n$. Certainly, due to Lemma~\ref{lem:ebound}, the truncation estimate is the same. The main estimate used in the proof of Theorem~\ref{th:errgj} is the bound on the stability function $r(z)$ and on $q(z)$. The additional terms in $\Ev_n(z)$ would only contribute to the constant. Similar comment holds for Theorem~\ref{th:errgj}.
\end{remark}




\section{Fast summation and computational cost}\label{sec:fastalg}

Now the efficient quadrature is available we explain how to use it to develop a fast algorithm for computing the corresponding discrete convolution.
In order to do this, we split the convolution as
\[
\sum_{j = 0}^n \omegav_j \fv_{n-j} =
\sum_{j = 0}^{n_0} \omegav_j \fv_{n-j}+\sum_{j = n_0+1}^n \omegav_j \fv_{n-j} = I^1_n + I^2_n,
\]
where as before $(\fv_n)_{\ell}=f(t_{n,\ell})$; see \eqref{eq:rkcq_gen}.
The first term is  computed exactly, whereas  for the second we can use the quadrature. Let $N_Q$ be the total number of quadrature nodes and let  $(w_k,x_k)$ denote the quadrature weights and nodes with the weights including  the values of $x_k^{-\alpha}$  in the region $L_0$ to $L$ where Gauss quadrature is used. Then  our approximation of $I^2_n$ has the form
\[
\sum_{j = n_0+1}^n \omegav_j \fv_{n-j}
\approx \sum_{k = 1}^{N_Q} w_k  (r(-\dt x_k))^{n_0+1} \sum_{j = 0}^{n-n_0-1}(r(-\dt x_k))^j \qv(-\dt x_k) \fv_{n-n_0-1-j}.
\]
Defining
\begin{equation}\label{Qs}
Q_{n,k} = \sum_{j = 0}^{n-n_0-1}(r(-\dt x_k))^j \qv(-\dt x_k) \fv_{n-n_0-1-j}
\end{equation}
we see that
\[
Q_{n,k} = r(-\dt x_k)Q_{n-1,k}+\qv(-\dt x_k)\fv_{n-n_0-1}, \qquad Q_{n_0,k} = 0.
\]
Hence the convolution can be approximated as
\[
\sum_{j = 0}^n \omegav_j \fv_{n-j} \approx
\sum_{j = 0}^{n_0} \omegav_j \fv_{n-j}
+\sum_{k = 1}^{N_Q} w_k  (r(-\dt x_k))^{n_0+1} Q_{n,k},
\]
with the $Q_{n,k}$ satisfying the above recursion. Notice that for each $k=1, \dots, N_Q$, $Q_{n,k}$ is the RK approximation at time $t_{n-n_0-1}$ of the ODE:
$$
\dot{q} = -x_k q + f, \qquad q(0)=0.
$$

Thus, from one step to the next one we only need updating $Q_{n,k}$, for $k=1,\dots,N_Q$,  $N_Q$ being the total number of quadrature nodes. Set $\veps$ the target accuracy of the quadrature. Then, from the results in Section~\ref{sec:quadcqw} it follows that the total computational cost is $O(N N_Q)$ with
\begin{equation}\label{totalnq}
N_Q=O(|\log(\veps)| \log(L/L_0)).
\end{equation}

For $n\ge 5$, Corollary~\ref{coro:L} implies $L\sim \dt^{-1}$ and from Theorem~\ref{th:errgj} a reasonable choice for $L_0$ is $L_0=4/(eT)$, which leads to
$$
N_Q = O(|\log(\veps)| \log(\dt^{-1}T)) = O(|\log(\veps)| \log(NT)).
$$

Therefore, the computational complexity is $O(|\log \veps|N \log N )$, whereas the storage requirement scales as $O(N_Q) = O(|\log \veps|\log N )$.
%

\section{Numerical experiments}
Given a tolerance $\tol > 0$, time step $\dt > 0$, minimal index $n_0$, final time $T > 0$, and the fractional power $\alpha \in (0,1)$  we use the above estimates to choose the parameters in the quadrature.

In particular we choose  $L_0 = 4/T$ and $L = A \dt^{-1}$ with $A$ such that the upper bound for the trunction error in \eqref{eq:remainder_intw} is less than $\tol/3$. We set $\tilde B = 3$ and
\[
J = \left\lfloor \frac{\log(L/L_0)}{\log(1+\tilde B)}\right\rfloor \text{ and } B = (L/L_0)^{1/J}-1.
\]
Note that in general this choice results in fewer integration intervals than when fixing $B$ and setting $J$ to the smallest integer such that $L \leq L_0(1+B)^J$.
Next, we set $L_j = L_0(1+B)^j$, for $j = 0,\dots,J$ and let $Q_0$ denote the number of quadrature points in the Gauss-Jacobi quadrature on $[0, L_0]$ and $Q_j$, $j = 0,\dots, J-1$, the number of Gauss quadrature points in the interval $[L_{j},L_{j+1}]$. We choose the smallest $Q_0$ so that the bound on $\|\tauv_{\mathrm{GJ},n}(Q_0)\|$ in Theorem~\ref{th:errgj} is less than $\tol/3$; note that in all of the experiments below we had $\vrho_{\text{opt}} < \vrho_{\text{max}}$. By doing a simple numerical minimization on the bound in Theorem~\ref{th:gerr_wj}, we find the optimal $Q_j$ such that the error $\|\tauv_{n,j}(Q_j)\| < \tol\, J^{-1}/3$. With this choice of parameters each weight $\omegav_j$, $j > n_0$, is computed to accuracy less than $\tol$.

In Figure~\ref{fig:weights_err} we show the error $\|\widetilde \omegav_n -\omegav_n\|$, where $\widetilde \omegav_n$ is the $n$th weight computed using the new quadrature scheme and $\omegav_n$ is an accurate approximation of the weight computed by standard means. We see that the error is bounded by the tolerance and that for the initial weights the error is close to this bound. The error for larger $n$ is considerably smaller than the required tolerance. This is expected, as in Corollary~\ref{coro:L} we need to use the worst case $n = n_0+1$ to determine the trunction parameter $A$.

We also investigate the number of quadrature points in dependence on $\dt$, $T$ in Table~\ref{tab:dtT}  and on $\alpha$ and $\tol$ in Table~\ref{tab:tol}. We observe only a moderate increase with decreasing $\dt$, $\tol$ and increasing $T$. The dependence on $\alpha$ is mild.

\begin{table}
  \centering
  \begin{tabular}{c|ccccc}
$\dt\big\backslash T$ & $1$ & $10$ & $100$ & $1000$ \\\hline
 $10^{-1}$ &$20$ & $30$ & $40$ & $49$ \\
$ 10^{-2}$ &$27$ & $36$ & $44$ & $52$ \\
$ 10^{-3}$ &$31$ & $39$ & $46$ & $50$ \\
$ 10^{-4}$ &$34$ & $40$ & $45$ & $48$
  \end{tabular}\hspace{1cm}
    \begin{tabular}{c|ccccc}
$\dt\big\backslash T$ & $1$ & $10$ & $100$ & $1000$ \\\hline
 $10^{-1}$ & $13$ & $24$ & $34$ & $44$ \\
$ 10^{-2}$ & $21$ & $31$ & $39$ & $46$ \\
$ 10^{-3}$ & $28$ & $35$ & $41$ & $46$ \\
$ 10^{-4}$ & $31$ & $37$ & $43$ & $45$
  \end{tabular}
  \caption{\small Dependence of the total number of quadrature points on time step $\dt$ and final time $T$. The other parameters are fixed at $n_0 = 5$, $B = 3$, $\tol = 10^{-6}$, $\alpha = 0.5$. On the left the data is for backward Euler and on the right for the 2-stage Radau IIA CQ.}
    \label{tab:dtT}
\end{table}

\begin{table}
  \centering
  \begin{tabular}{c|cccccccccc}
    $\tol\big\backslash\alpha$ & $0.1$ &  $0.3$ &  $0.5$  & $0.7$  & $0.9$ \\ \hline
$10^{-2}$ & $11$ & $11$ & $10$ & $8$ & $6$ \\
$10^{-4}$ & $27$ & $27$ & $26$ & $25$ & $21$ \\
$10^{-6}$ & $45$ & $44$ & $45$ & $43$ & $36$ \\
$10^{-8}$ & $66$ & $65$ & $64$ & $61$ & $55$ \\
$10^{-10}$ &$86$ & $87$ & $85$ & $82$ & $74$
  \end{tabular}\hspace{.5cm}
  \begin{tabular}{c|cccccccccc}
$\tol\big\backslash\alpha$ & $0.1$ &  $0.3$ &  $0.5$  & $0.7$  & $0.9$ \\ \hline
$10^{-2}$ & $9$ & $9$ & $8$ & $8$ & $6$ \\
$10^{-4}$ & $23$ & $25$ & $24$ & $23$ & $20$ \\
$10^{-6}$ & $39$ & $39$ & $39$ & $37$ & $35$ \\
$10^{-8}$ & $71$ & $68$ & $65$ & $53$ & $51$ \\
$10^{-10}$ &$96$ & $93$ & $90$ & $86$ & $77$
  \end{tabular}
  \caption{\small Dependence of the total number of quadrature points on the tolerance $\tol$ and the fractional power $\alpha$. The other parameters are fixed at $\dt = 10^{-2}$, $T = 50$, $n_0 = 5$, $B = 3$. Again the data on the left is for backward Euler and on the right for 2-stage Radau IIA.}
  \label{tab:tol}
\end{table}

\begin{figure}
  \centering
  \includegraphics[width=0.7\textwidth]{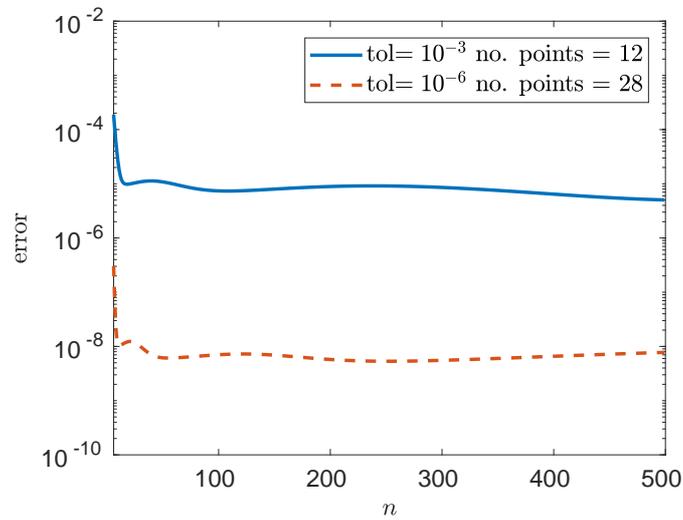}
  \caption{\small We show the error in the computation of the 2-stage Radau IIA weights $\omega_n$ for $n > n_0$ with two different tolerances. The number of quadrature points is also shown. The results are for $\alpha = 0.5$, $T = 5$, $\dt = 10^{-2}$, $n_0 = 5$, and $B = 3$. }
  \label{fig:weights_err}
\end{figure}

\subsection{Fractional integral}

Let us now consider the evaluation of a fractional integral
\begin{equation}
  \label{eq:frac_int_num}
  u(t) = \mathcal{I}^{\alpha}[g](t),
\end{equation}
where
\[
  g(t) = t^3 e^{-t}.
\]
  First we investigate the behaviour of the standard implementation of CQ, based on FFT.  In the particular case of the two-stage Radau IIA, $p = 3$ and $q = 2$, from Theorem~\ref{th:LuOs} we would expect full order convergence.  We set  $T = 128$,  and $\dt = 2^{3-j}$, $j = 0, \dots, 7$, $\alpha = 1/4$. We do not have access to the exact solution $u(t)$, so its role is taken by an accurate numerical approximation. In Figure~\ref{fig:convcq} we show the convergence of the error $\max_n |u(t_n)-u_n|$ using the standard implementation of CQ. We compare it with  the  theoretical reference curve $ 10^{-2.5}(\dt^3+ \left|\log(\dt)\right| \dt^{3+\alpha})$, which fits the results better in this pre-asymptotic regime than the dominant term  $\dt^3$ on its own.

  \begin{figure}
    \centering
    \includegraphics[width=.7\textwidth]{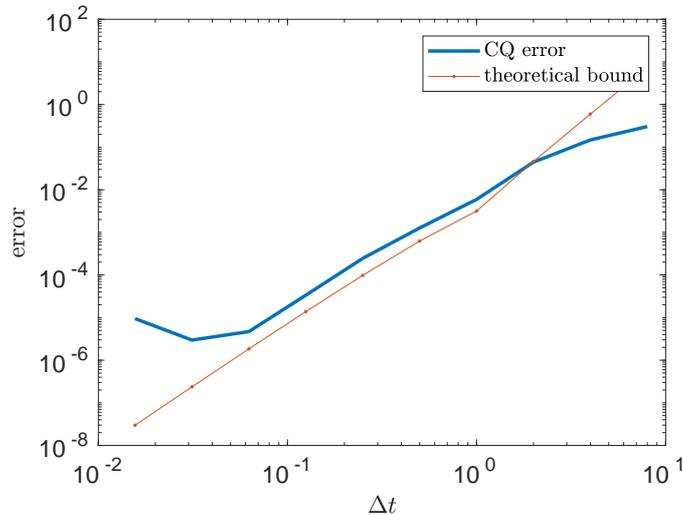}
    \caption{\small Convergence of the error $\max_n |u(t_n)-u_n|$ for the 2-stage Radau IIA convolution quadrature of the fractional integral \eqref{eq:frac_int_num}.}
    \label{fig:convcq}
  \end{figure}

    \begin{figure}
    \centering
    \includegraphics[width=.49\textwidth]{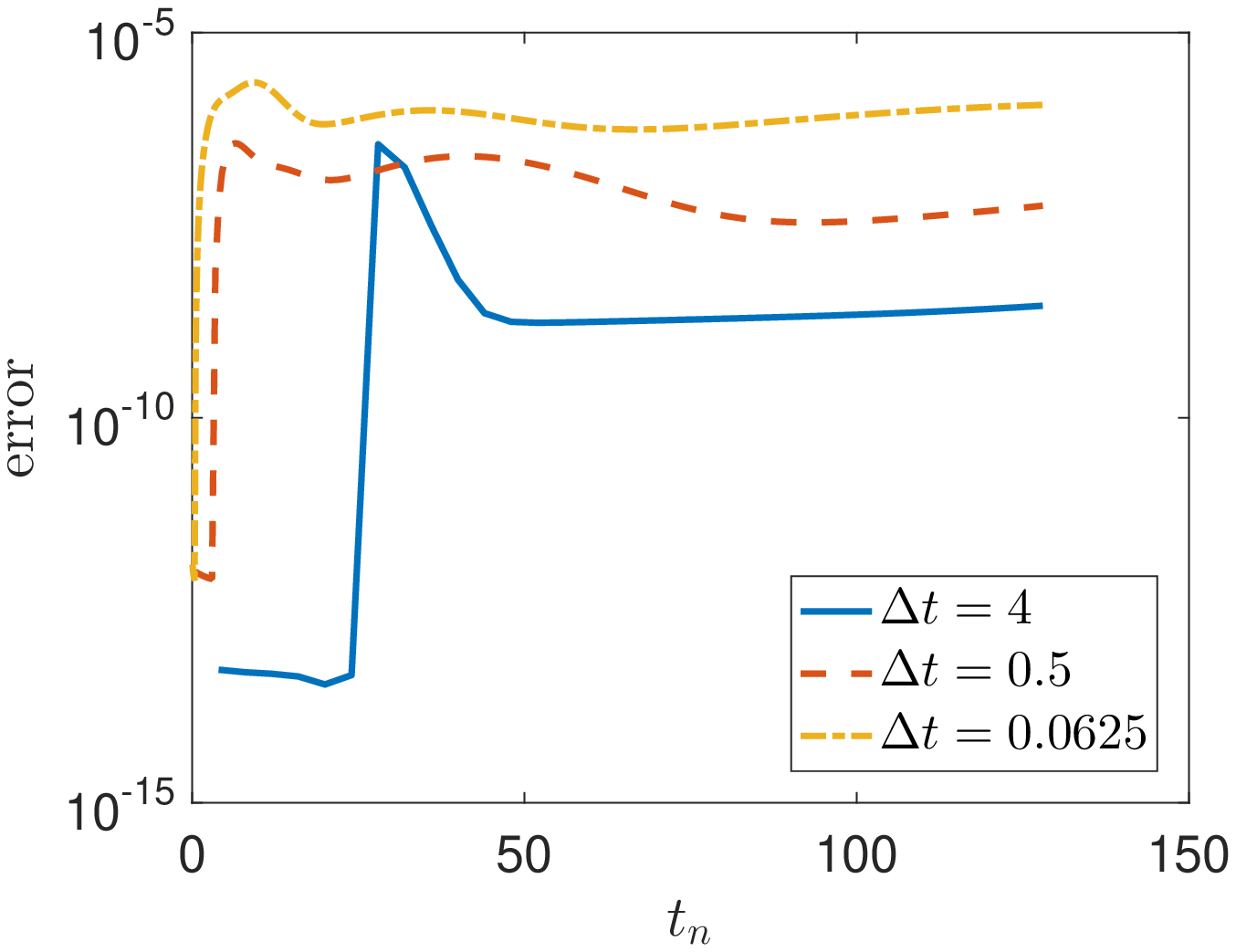}
        \includegraphics[width=.49\textwidth]{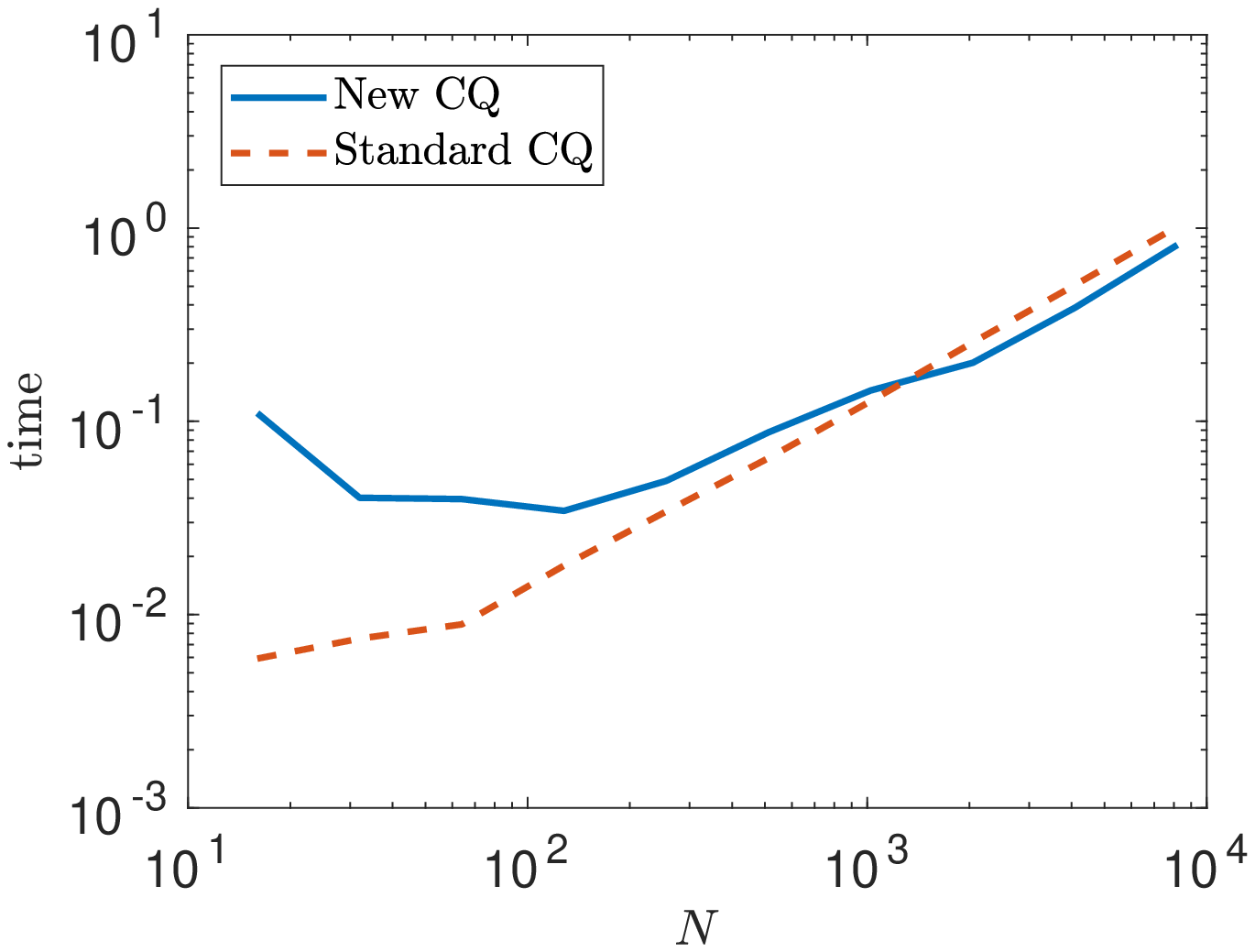}
    \caption{\small We plot in the left graph the difference $|\tilde u_n-u_n|$ against $t_n$, where $u_n $is computed using the standard implementation of CQ and $\tilde u_n$ with the new method at $\tol = 10^{-6}$. On the right we plot the required time for the two methods.}
    \label{fig:cqdiff}
  \end{figure}

  Next, we apply our new quadrature implementation of CQ. We set  $\tol = 10^{-6}$, and the rest of the parameters as in the above Section. We denote by $\tilde u_n$ the new approximation of $u_n$ and plot the error $|u_n - \tilde u_n|$ in Figure~\ref{fig:cqdiff}. We see that the error is bounded by $10^{-6}$ for all $n$, showing that the final perturbation error introduced by our approximation of the CQ weights remains  bounded with respect to the target accuracy in our quadrature, cf.~\cite{SchaLoLu}. We also compare computational times in Figure~\ref{fig:cqdiff}. For the implementation of the standard CQ we have used the $O(N\log N)$ FFT based algorithm from \cite{Lu_88II}. We see that for larger time-steps the FFT method is faster due to a certain overhead in constructing the quadrature points for the new method. For smaller time steps however the new method is even marginally faster. The main advantage of the new method is the $O(\log N)$ amount of memory required compared to $O(N)$ amount of memory by the standard method. For example, in this computation with the smallest time step, there are $N = 2048$ time steps and the total number of quadrature points is $37$. As each quadrature point carries approximately the same amount of memory as one directly computed time-step, we see that the memory requirement is around 50 times smaller with the new method for this example. Such a difference in memory requirements becomes of crucial importance when faced with non-scalar examples coming from discretizations of PDE. The next section considers this case.

\section{Application to a fractional diffusion equation}\label{sec:FDE}

We now consider the problem of finding $u(t) \in H_0^1(\Omega)$ such that
\begin{equation}  \label{eq:heat}
\begin{aligned}
    \partial_t^\beta u - \Delta u &= f, & &\text{for } (x,t) \in \Omega \times [0,T],\\
    u^{(k)}(x,0) &= 0, & &\text{for } x\in \Omega, \; k = 0,\dots, m-1,
\end{aligned}
\end{equation}
with $\beta \in (0,2) \setminus \{1\}$ and $m = \lceil \beta \rceil$. Here, $\Omega$ is a bounded, convex  Lipschitz domain in $\bR^d$, $d = 1,2,3$,   $H^1_0(\Omega)$ the Sobolev space of functions with zero trace,  and $\partial_t^\beta$ the fractional derivative
\begin{equation} \label{eq:frac_der}
\partial_t^\beta  u:= \mathcal{I}^{m-\beta}[\partial^m_t u](t)
= \frac{1}{\Gamma(1-m+\beta)}\int_0^t (t-s)^{m-\beta-1} \partial^{m}_t u(s)\,ds.
\end{equation}
This is the fractional derivative in the Caputo sense, which in the case $u^{(k)}(0) = 0$, $k = 0,\dots, m-1$,  is equivalent to the Riemann-Liouville derivative.

\begin{remark}
For simplicity  we avoid the integer case $\beta = 1$ as it is just the standard heat equation and in some places this case would have to be treated slighlty differently.
\end{remark}

The application of CQ based on BDF2 to integrate \eqref{eq:frac_der}  in time has been analyzed in \cite[Section 8]{CuLuPa}. A related problem with a fractional power of the Laplacian has been studied in \cite{NoOSa}, but not with a CQ time discretization. Here we apply Runge--Kutta based CQ. The analysis of the application of RK based CQ to \eqref{eq:heat} is not available in the literature, hence we give the analysis here for sufficiently smooth and compatible right-hand side $f$. We first analyze the error of the spatial discretization.

\subsection{Space-time discretization of the FPDE: error estimates}

Let $X_{\Delta x} \subset H^1_0(\Omega)$ be a finite element space of piecewise linear functions and let $\Delta x$ be the meshwidth. Applying the Galerkin method in space we obtain a system of fractional differential equations: Find $\uv(t) \in X_{\Delta x}$ such that
\begin{equation}  \label{eq:heat_sd0}
\begin{aligned}
    \int_\Omega \partial_t^\beta \uv(t) \vv  + \nabla \uv(t)\; \nabla \vv \; dx &= \int_\Omega f(t) \vv dx, & &\text{for } t \in [0,T], \vv \in X_{\Delta x}\\
    \uv^{(k)}(0) &= \mathbf{0}, & &\text{for } x\in \Omega, \; k = 0,\dots, m-1,
  \end{aligned}
\end{equation}
\begin{theorem}\label{th:semi_disc}
  Let $f \in C^m([0,T]; L^2(\Omega))$ with $f^{(k)}(0) = 0$, $k = 0,\dots,m-1$ and  let $\uv(t)$ be the solution of \eqref{eq:heat_sd0} and  $u(t)$ the solution of \eqref{eq:heat}. Then if $m > \beta$ we have
  \[
        \|u(t) -\uv(t)\|_{H^1(\Omega)} \leq C \Delta x \int_0^t \|f^{(m)}(\tau)\|_{L^2(\Omega)}d\tau.
      \]
      If further $m > 2\beta$ we have
        \[
        \|u(t) -\uv(t)\|_{L^2(\Omega)} \leq C (\Delta x)^2 \int_0^t \|f^{(m)}(\tau)\|_{L^2(\Omega)}d\tau.
      \]
    \end{theorem}
    \begin{proof}
Consider the Laplace transform of \eqref{eq:heat}
      \begin{equation}
    \label{eq:heat_s}
    z^\beta \hat u -\Delta \hat u = \hat f, \qquad |\arg(z)| <  \min(\pi, (\pi-\delta)/\beta),
  \end{equation}
 for some fixed $\delta > 0$,   and
the bilinear form
  \[
    a(u,v) =  \int_\Omega z^\beta u \overline v dx +  \int_\Omega \nabla u \cdot \nabla \overline v dx.
  \]
Hence,  $a(u,v) = \int_\Omega \hat f \overline v dx$ is  the weak form of \eqref{eq:heat_s}. The bilinear form   is continuous
  \[
    |a(u,v)| \leq \max(1,|z|^{\beta}) \|u\|_{H^1(\Omega)}  \|v\|_{H^1(\Omega)}
  \]
  and
  \[
    \Re a(z^{-\beta}u,u) =  \|u\|_{L^2(\Omega)}^2 +  \Re z^{-\beta}\|\nabla u\|_{L^2(\Omega)}^2\geq \Re z^{-\beta}\|\nabla u\|_{L^2(\Omega)}^2
  \]
  and
  \[
    |\Im a(z^{-\beta} u,u)| = |\Im z^{-\beta}| \|\nabla u\|_{L^2(\Omega)}^2.
  \]
  Hence
  \[
    |a(u,u)|  = |z|^\beta |a(z^{-\beta}u,u)|\geq \left\{
      \begin{array}{cc}
        \|\nabla u\|_{L^2(\Omega)}^2& \text{ if }\Re z^{-\beta} > 0, \\
        |z|^\beta |\Im z^{-\beta}|\|\nabla u\|_{L^2(\Omega)}^2 & \text{ otherwise}.

      \end{array}
    \right.
    \]
As $|\arg(z^\beta )| < \pi-\delta$, we have that  $|a(u,u)| \geq C\|\nabla u\|_{L^2(\Omega)}^2$ and using the Poincar\'e inequality we obtain coercivity in $H^1_0(\Omega)$.  Lax-Milgram gives us that there exists a unique $\hat u \in H^1_0(\Omega)$ solution of \eqref{eq:heat_s} and that
  \[
    \|\hat u\|_{H^1(\Omega)} \leq C_\Omega \|\hat f\|_{H^{-1}(\Omega)}.
  \]
  If furthermore $\hat f \in L^2(\Omega)$, we have that
    \[
    |\Im z^\beta|   \|\hat u\|_{L^2(\Omega)}^2 =  \left|\Im \int_\Omega \hat f\overline{\hat u} dx\right| \leq \|\hat f\|_{L^2(\Omega)}\|\hat u\|_{L^2(\Omega)}
  \]
and
\[
    \Re z^\beta   \|\hat u\|_{L^2(\Omega)}^2 =  -\|\nabla u\|_{L^2(\Omega)}^2+\Re \int_\Omega \hat f\overline{\hat u} dx \leq  \|\hat f\|_{L^2(\Omega)}\|\hat u\|_{L^2(\Omega)}.
\]
  Dividing by $|z|^\beta\|\hat u\|_{L^2(\Omega)}$  and using the fact that $|\arg(z^\beta )| < \pi-\delta$ gives
  \begin{equation}
    \label{eq:res_est}
    \|\hat u\|_{L^2(\Omega)} \leq C |z|^{-\beta} \|\hat f\|_{L^2(\Omega)}.
  \end{equation}

For the finite element solution, C\'ea's lemma gives that
    \[
      \|\hat u - \hat \uv\|_{H^1(\Omega)} \leq \max(1,|z|^{\beta})\inf_{\vv \in X_{\Delta x}}\|\hat u - \vv\|_{H^1(\Omega)},
    \]
    where $\hat \uv$ denotes the Laplace transform of $\uv$.    Using the  Aubin-Nitche trick, we obtain the estimate in the weaker norm
    \[
      \|\hat u - \hat \uv\|_{L^2(\Omega)} \leq C \max(1,|z|^{\beta})\|\hat u - \hat \uv\|_{H^1(\Omega)} \sup_{g \in L^2(\Omega),\; g \neq 0}  \inf_{\vv \in X_{\Delta x}}\frac{ \|\varphi_g-\vv\|_{H^1(\Omega)}}{\|g\|_{L^2(\Omega)}},
      \]
      where $\varphi_g$ is the solution of the dual problem
      \[
        a(\vv,\varphi_g) = \int_\Omega g \overline \vv \; dx, \qquad \text{ for all } \vv \in X_{\Delta x}.
      \]
      Recalling that  $\Omega$ is assumed to be convex, we can use  standard elliptic regularity results together with $-\Delta \hat u = \hat f-z^{\beta}\hat u$ to show that
      \[
        \|\hat u \|_{H^2(\Omega)} \leq C \|\hat f-z^{\beta}\hat u\|_{L^2(\Omega)}
        \leq C(\|\hat f\|_{L^2(\Omega)}+ \|z^\beta \hat u\|_{L^2(\Omega)})
        \leq C \|\hat f\|_{L^2(\Omega)},
      \]
where the final inequality follows from \eqref{eq:res_est}.
      Similarly
      \[
        \| \varphi_g \|_{H^2(\Omega)}  \leq C \|g \|_{L^2(\Omega)}
      \]
      and using standard approximation results we have that
      \[
        \|\hat u - \hat \uv\|_{H^1(\Omega)} \leq C \Delta x\max(1,|z|^{\beta})  \|\hat f\|_{L^2(\Omega)} =
        C \Delta x\max(|z|^{-m},|z|^{\beta-m})  \|z^m\hat f\|_{L^2(\Omega)}
        \]
        and
      \[
        \|\hat u - \hat \uv\|_{L^2(\Omega)} \leq C \max(|z|^{-m},|z|^{2\beta-m}) (\Delta x)^2 \|z^m\hat f\|_{L^2(\Omega)}.
      \]
      The proof is completed by applying Parseval's theorem.
    \end{proof}

    The fully discrete system is now obtained by simply discretizing the fractional derivative at stage level using RK-CQ:
    \begin{equation}  \label{eq:heat_fd0}
  \int_\Omega (\ptv_t^h)^\beta \Uv(\tv_n) \vv  + \nabla \Uv_n\; \nabla \vv \; dx = \int_\Omega f(\tv_n) \vv dx,
\end{equation}
for $n = 1,\dots, N-1, \vv \in X_h$.

\begin{theorem}\label{th:fully_disc}
  Let an $A$-stable, $s$-stage Runge-Kutta method of order $p$ and stage order $q$ be given which satisifes the assumptions of Section~\ref{subsec:rkcq} and let $u(t)$ be the solution of \eqref{eq:heat} and $\Uv$ solution of \eqref{eq:heat_fd0}. If $\uv^h$ denotes the solution at full time steps, i.e., $\uv^h_{n+1} = \Uv_{n,s}$ and if $f \in C^p([0,T]; L^2(\Omega))$ with  $f^{(k)}(0) = 0$, for $k = 0,\dots, \lceil \beta \rceil-1$ then
  \[
    \begin{split}
      \|u(t_n)-\uv^h_n\|_{L^2(\Omega)} = O(\Delta x^2)&+O(h^p+h^{q+1+\beta})\\
      &+O(h^p) \left( \sum_{\ell=0}^q \left(1+t_n^{\beta+\ell-p} \right) \|f^{(\ell)}(0)\|_{L^2(\Omega)} \right).
          \end{split}
    \]
\end{theorem}
\begin{proof}
  Denote by $K(z): f \mapsto u$ the solution operator of the Laplace transformed problem \eqref{eq:heat_s} and note the resolvent estimate
  \[
    \|K(z)\|_{L^2(\Omega) \leftarrow L^2(\Omega)} \leq C|z|^{-\beta}, \qquad |\arg (z )| < \min(\pi, (\pi-\delta)/\beta),
  \]
for any $\delta > 0$,  following from \eqref{eq:res_est}; see also \cite{Paz} and \cite{CuLuPa}. The same estimate holds for the solution operator $K_{\Delta x}: f \mapsto \uv$ of the Galerkin discretization in the space $X_{\Delta x}$. Note also that
  \[
    u(t) = K(\partial_t) f, \quad
    \uv(t) = K_{\Delta x}(\partial_t) f, \quad
    \Uv =  K_{\Delta x}(\ptv_t^{h}) f,\quad     \uv^h =  K_{\Delta x}(\partial_t^{h}) f.
  \]
  The second to last equality above follows from standard properties of convolution quadrature, see for example \cite[Section~4]{Lub94} and \cite[Chapter~9]{Say}, whereas the last one is simply the use of operational notation explained in Remark~\ref{rem:notation}. The result now follows from Theorem~\ref{th:LuOs},  Theorem~\ref{th:semi_disc}, and the triangle inequality.
\end{proof}

\subsection{Implementation and numerical experiments}\label{subsec:algFDE}
Though all the information needed for the implementation is given in the preceding pages, for the benefit of the reader we give some more detail here. Let $M$ denote the number of degrees of freedom in space, i.e., $M = \dim X_{\Delta x}$, let $\Bv$ and $\Av$ be the mass and stiffness matrices.


For simplicity of presentation we assume $\beta \in (0,1)$ and let $\Uv_n \in \bR^{sM}$ now denote the vector   $\Uv_n = [\Uv_{n,1},\dots, \Uv_{n,s}]^T$ with  $\Uv_{n,\ell} \approx \uv(t_{n,\ell})$, $\ell = 1,\dots, s$. Hence the fully discrete system can be written as a system of linear equations
    \begin{equation}
  \label{eq:heat_fd}
  \begin{aligned}
    \left((\ptv_t^h)^\beta\otimes \Bv  \right) \Uv  + \left( \Iv_{s}\otimes \Av\right) \Uv &= \Fv, & &\text{for } \tv_n \in  [0,T],
  \end{aligned}
\end{equation}
where $\Iv_D$ denotes the identity matrix of size $D \times D$ and $\Fv_j\in \bR^{sM}$.

Note that the composition rule allows us to write the CQ approximation to  $\partial_t^{\beta} y$, as
\[
   (\ptv_t^h)^\beta \Yv =      (\ptv_t^h)^{\beta-m}(\ptv_t^h)^m \Yv,
\]
with  $(\Yv_j)_{\ell} \approx y(t_{j,\ell})$ and $m = \lceil \beta \rceil$. As $\beta -m < 0$, the discrete version $(\ptv_t^h)^{\beta-1}$ of the fractional integral $\partial_t^{\beta-1} = \mathcal{I}^{1-\beta}$ can be evaluated by our fast algorithm, whereas  $(\ptv_t^h)^m$ is the standard one-step Runge-Kutta approximation of the derivative repeated $m$ times.

For simplicity of presentation  introduce  new variables $\Vv_j \in \bR^{M s}$ with
\[
  \Vv_n  =  \left(\ptv_t^h \otimes \Iv_M \right)\Uv(\tv_n).
\]
Note that
\[
  \Vv_n  = \left( \Iv_M \otimes {\bf D}_0  \right) \Uv_n+ \left( \Iv_M \otimes {\bf D}_1  \right) \Uv_{n-1}, \quad n=1,\dots,N,
\]
where from \eqref{eq:rksymbol} we have that
\[
  \Dv_0 = \frac1{\dt} \A^{-1}, \qquad
   \Dv_1 = \frac1{\dt} \A^{-1}\bone \bv^T \A^{-1}.
 \]

 Then the fully discrete system \eqref{eq:heat_fd} becomes
  \[
    \sum_{j = 0}^n \left( \Wv_{n-j} \otimes \Iv_M \right) \Vv_j + \left( \Iv_s\otimes \Bv  \right)^{-1} \left( \Iv_s \otimes \Av \right) \Uv_n = \left(  \Iv_s \otimes \Bv \right)^{-1} \Fv_n,
  \]
  where $\Wv_j$ are the weight matrices for the fractional integral $\mathcal{I}^\alpha$ with $\alpha = 1-\beta$.
  Rearranging terms so that the known vectors are on the right-hand side and denoting
  $$
 \mathcal{A}=\Iv_s\otimes \Av, \qquad  \mathcal{B}=\Iv_s \otimes \Bv,
  $$
  we obtain
  \[
    \left(\Wv_0 \otimes \Iv_M \right) \Vv_n+  \mathcal{B}^{-1}\mathcal{A} \Uv_n =  \mathcal{B}^{-1} \Fv_n  - \sum_{j = 0}^{n-1} \left( \Wv_{n-j} \otimes \Iv_M \right) \Vv_j
  \]
  or  using the definition of $\Vv_n$
  \[
    \left( \Wv_0 \Dv_0\otimes \Iv_M \right) \Uv_n+ \mathcal{B}^{-1}\mathcal{A} \Uv_n =  - \left(\Wv_0 {\bf D}_1 \otimes \Iv_M \right) \Uv_{n-1}- \mathcal{B}^{-1} \Fv_n  - \sum_{j = 0}^{n-1} \left(\Wv_{n-j}\otimes \Iv_M \right) \Vv_j.
  \]
  At each time step this system needs to be solved, where the expensive part  is the computation of the discrete convolution in the right-hand side and the storage of all the vectors $\Vv_j$. This problem is resolved by our fast method of evaluation of discrete convolutions with the following variation with respect to Section~\ref{sec:fastalg} in order to deal with the stages:
  \begin{align}
   \sum_{j = n_0+1}^{n} \left( \Wv_{j} \otimes \Iv_M \right) \Vv_{n-j}
  \approx \dt\sum_{k = 1}^{N_Q} w_k  (r(-\dt x_k))^{n_0} \Qv_{n-1,k}
  \end{align}
  with
  $$
  \Qv_{\ell,k}=\sum_{j = 0}^{\ell-n_0-1}(r(-\dt x_k))^{j}\left(  \left( \Iv_s + \dt x_k \A \right)^{-1} \bone \qv(-\dt x_k)  \otimes \Iv_M \right) \Vv_{\ell-n_0-1-j}
  $$
  satisfying the recursion
  $$
  \Qv_{\ell,k}= r(-\dt x_k)\Qv_{\ell-1,k}+ \left( \left( \Iv_s+ \dt x_k \A \right)^{-1} \bone \qv(-\dt x_k)\otimes \Iv_M \right) \fv_{\ell-n_0-1}, \quad \Qv_{n_0,k}=0.
  $$
  As a final point let us note that due to \eqref{powomega}
  \[
    \Wv_0 = \left( \frac{\Deltav(0)}{\dt} \right)^{-\alpha} = \dt^\alpha \A^{\alpha}
  \]
  and hence
  \[
    \Wv_0 \Dv_0 =  (\dt)^{\alpha-1}\A^{\alpha-1} = (\dt)^{-\beta}\A^{-\beta}.
    \]
As the spectrum of $\A^{-\beta}$ is in the right-half complex plane the problem to be solved at each time-step has a unique solution.

For the numerical experiments we let $\Omega$ be the square with corners $(-1,-1)$ and $(1,1)$ and choose $f$ so that the exact solution is
  \[
    u(x,t) = \sin^3 \left(\tfrac32\pi t \right)\cos \left(\tfrac12\pi x_1 \right) \cos\left(\tfrac12\pi x_2\right).
 \]
We let the final time be $T = 7$, fix the finite element space on a triangular mesh with meshwidth $\Delta x = 5 \times 10^{-3}$ and compute the error in the $L^2(\Omega)$ norm at $t = T$. The error and memory requirements as the number of time-steps is increased are given in  Table~\ref{tab:heat} for our new method and for the standard implementation of the CQ. We have used as tolerance $\tol = 10^{-4}$ and the 2-stage RadauIIA based CQ, for which the theory predicts convergence of order $O(\dt^3)$. We see that the error is the same for the two implementations of the CQ, achieving in both cases the predicted order 3, but that the memory requirements for the new method stay almost constant whereas for the standard implementation they grow linearly.



        \begin{table}
      \centering
      \begin{tabular}{c|c|c|c|c}
        $N$ & error & memory (MB) & standard err.\ & standard mem.\ (MB)\\\hline
        32 & $2.94 \times 10^{-1}$&39.1 & $2.94 \times 10^{-1}$ & 59.2\\
        64 & $3.07 \times 10^{-2}$ & 40.3 & $3.07 \times 10^{-2}$ & 98.7\\
        128 & $2.61\times 10^{-3}$ & 42.8 & $2.61\times 10^{-3}$ & 177.6\\
        256 & $2.98\times 10^{-4}$   & 44.0 & $3.01\times 10^{-4}$  & 335.4
      \end{tabular}
      \caption{\small We show the error and the memory requirements for the new method and the standard implementation of CQ.}
      \label{tab:heat}
    \end{table}


\begin{thebibliography}{10}

\bibitem{Suecos_2008}
K.~Adolfsson, M.~Enelund, and S.~Larsson.
\newblock Space-time discretization of an integro-differential equation
  modeling quasi-static fractional-order viscoelasticity.
\newblock {\em J. Vib. Control}, 14(9-10):1631--1649, 2008.

\bibitem{Ba18}
D.~Baffet.
\newblock A {G}auss-{J}acobi kernel compression scheme for fractional
  differential equations.
\newblock {\em arXiv:1801.06095}, 2018.

\bibitem{BaHe17}
D.~Baffet and J.~S. Hesthaven.
\newblock A kernel compression scheme for fractional differential equations.
\newblock {\em SIAM J. Numer. Anal.}, 55(2):496--520, 2017.

\bibitem{Ban10}
L.~Banjai.
\newblock Multistep and multistage convolution quadrature for the wave
  equation: Algorithms and experiments.
\newblock {\em SIAM J. Sci. Comput.}, 32(5):2964--2994, 2010.

\bibitem{BanLu_2011}
L.~Banjai and C.~Lubich.
\newblock An error analysis of {R}unge-{K}utta convolution quadrature.
\newblock {\em BIT}, 51(3):483--496, 2011.

\bibitem{BanLM}
L.~Banjai, C.~Lubich, and J.~M. Melenk.
\newblock Runge-{K}utta convolution quadrature for operators arising in wave
  propagation.
\newblock {\em Numer. Math.}, 119(1):1--20, 2011.

\bibitem{lb_ms}
L.~Banjai and M.~Schanz.
\newblock Wave propagation problems treated with convolution quadrature and
  {BEM}.
\newblock In U.~Langer, M.~Schanz, O.~Steinbach, and W.~L. Wendland, editors,
  {\em Fast Boundary Element Methods in Engineering and Industrial
  Applications}, volume~63 of {\em Lecture Notes in Applied and Computational
  Mechanics}, pages 145--184. Springer Berlin Heidelberg, 2012.

\bibitem{CuLuPa}
E.~Cuesta, C.~Lubich, and C.~Palencia.
\newblock Convolution quadrature time discretization of fractional
  diffusion-wave equations.
\newblock {\em Math. Comp.}, 75(254):673--696, 2006.

\bibitem{hals}
E.~Hairer, C.~Lubich, and M.~Schlichte.
\newblock Fast numerical solution of nonlinear {V}olterra convolution
  equations.
\newblock {\em SIAM J. Sci. Stat. Comput.}, 6(3):532--541, 1985.

\bibitem{HaWII}
E.~Hairer and G.~Wanner.
\newblock {\em Solving ordinary differential equations. {II}}, volume~14 of
  {\em Springer Series in Computational Mathematics}.
\newblock Springer-Verlag, Berlin, second edition, 1996.

\bibitem{Henrici_II}
P.~Henrici.
\newblock {\em Applied and computational complex analysis. {V}ol. 2}.
\newblock Wiley Interscience [John Wiley \& Sons], New York, 1977.

\bibitem{JiangZhang}
S.~Jiang, J.~Zhang, Q.~Zhang, and Z.~Zhang.
\newblock Fast evaluation of the {C}aputo fractional derivative and its
  applications to fractional diffusion equations.
\newblock {\em arXiv:1511.03453}, 2015.

\bibitem{Li}
J.-R. Li.
\newblock A fast time stepping method for evaluating fractional integrals.
\newblock {\em SIAM J. Sci. Comput.}, 31(6):4696--4714, 2009/10.

\bibitem{Lub_frac}
C.~Lubich.
\newblock Discretized fractional calculus.
\newblock {\em SIAM J. Math. Anal.}, 17(3):704--719, 1986.

\bibitem{Lu_88I}
C.~Lubich.
\newblock Convolution quadrature and discretized operational calculus {I}.
\newblock {\em Numer. Math.}, 52:129--145, 1988.

\bibitem{Lu_88II}
C.~Lubich.
\newblock Convolution quadrature and discretized operational calculus {II}.
\newblock {\em Numer. Math.}, 52:413--425, 1988.

\bibitem{Lub94}
C.~Lubich.
\newblock On the multistep time discretization of linear initial-boundary value
  problems and their boundary integral equations.
\newblock {\em Numer. Math.}, 67:365--389, 1994.

\bibitem{LubOst}
C.~Lubich and A.~Ostermann.
\newblock Runge-{K}utta methods for parabolic equations and convolution
  quadrature.
\newblock {\em Math. Comp.}, 60(201):105--131, 1993.

\bibitem{LuScha}
C.~Lubich and A.~Sch{\"a}dle.
\newblock Fast convolution for nonreflecting boundary conditions.
\newblock {\em SIAM J. Sci. Comput.}, 24(1):161--182, 2002.

\bibitem{NoOSa}
R.~H. Nochetto, E.~Ot\'arola, and A.~J. Salgado.
\newblock A {PDE} approach to space-time fractional parabolic problems.
\newblock {\em SIAM J. Numer. Anal.}, 54(2):848--873, 2016.

\bibitem{Paz}
A.~Pazy.
\newblock {\em Semigroups of linear operators and applications to partial
  differential equations}, volume~44 of {\em Applied Mathematical Sciences}.
\newblock Springer-Verlag, New York, 1983.

\bibitem{Say}
F.-J. Sayas.
\newblock {\em Retarded Potentials and Time Domain Boundary Integral
  Equations}, volume~50 of {\em Springer Series in Computational Mathematics}.
\newblock Springer, Heidelberg, 2016.

\bibitem{SchaLoLu}
A.~Sch{\"a}dle, M.~L{\'o}pez-Fern{\'a}ndez, and C.~Lubich.
\newblock Fast and oblivious convolution quadrature.
\newblock {\em SIAM J. Sci. Comput.}, 28(2):421--438, 2006.

\bibitem{Tre}
L.~N. Trefethen.
\newblock {\em Approximation theory and approximation practice}.
\newblock Society for Industrial and Applied Mathematics (SIAM), Philadelphia,
  PA, 2013.

\bibitem{Kar16}
Y.~Yu, P.~Perdikaris, and G.~E. Karniadakis.
\newblock Fractional modeling of viscoelasticity in 3{D} cerebral arteries and
  aneurysms.
\newblock {\em J. Comput. Phys.}, 323:219--242, 2016.

\bibitem{Yus}
S.~B. Yuste, L.~Acedo, and K.~Lindenberg.
\newblock Reaction front in an {$A+B \rightarrow C$} reaction-subdiffusion
  process.
\newblock {\em Phys. Rev. E}, 69:036126, Mar 2004.

\bibitem{Bu17}
F.~Zeng, I.~Turner, and K.~Burrage.
\newblock A stable fast time-stepping method for fractional integral and
  derivative operators.
\newblock {\em arXiv:1703.05480}, 2017.

\end{thebibliography}

\def\cprime{$'$}

\end{document}